\documentclass[reqno,11pt]{amsart}
\usepackage{a4wide,color,eucal,enumerate,mathrsfs}
\usepackage[normalem]{ulem}
\usepackage{amsmath,amssymb,epsfig,amsthm} 
\usepackage[latin1]{inputenc}
\usepackage{psfrag}
\usepackage{hyperref,cleveref,tikz}

\def\az{\alpha}

\def\dist{{\mathop\mathrm{\,dist\,}}}

\def\ez{\epsilon}

\def\gz{{\gamma}}

\def\bint{{\ifinner\rlap{\bf\kern.35em--}
\int\else\rlap{\bf\kern.45em--}\int\fi}\ignorespaces}

\def\bbint{{\ifinner\rlap{\bf\kern.35em--}
\hspace{0.078cm}\int\else\rlap{\bf\kern.45em--}\int\fi}\ignorespaces}

\def\diam{{\mathop\mathrm{\,diam\,}}}

\newcommand{\R}{\mathbb{R}}

\newtheorem{thm}{Theorem}[section]
\newtheorem{lem}[thm]{Lemma}
\newtheorem{prop}[thm]{Proposition}
\newtheorem{defn}[thm]{Definition}
\numberwithin{equation}{section}

\theoremstyle{remark}
\newtheorem{rem}[thm]{Remark}

\def\bint{{\ifinner\rlap{\bf\kern.35em--}
\int\else\rlap{\bf\kern.45em--}\int\fi}\ignorespaces}

\usepackage{graphicx}
\usepackage{import}
\usepackage{xifthen}
\usepackage{pdfpages}
\usepackage{transparent}

\title[Sobolev trace inequalities on John domains and its applications]{Sobolev trace inequalities on John domains and its applications}
\author{Weicong Su and Yi Ru-Ya Zhang}
\date{\today}
\address{State Key Laboratory of Mathematical Sciences, Academy of Mathematics and Systems Science, Chinese Academy of Sciences, Beijing 100190, China}
\address{Academy of Mathematics and Systems Science, the Chinese Academy of Sciences, Beijing 100190, P. R. China}
\email{suweicong@amss.ac.cn}  
\email{yzhang@amss.ac.cn}

\thanks{Both of the authors are funded by the National Key R\&D Program of China (Grant No. 2025YFA1018400 \&  No. 2021YFA1003100), NSFC Grant No. 12288201 \& No. 12571128, the Chinese Academy of Sciences, and CAS Project for Young Scientists in Basic Research, Grant No. YSBR-031.
}

\begin{document}

\begin{abstract}
We prove that a trace inequality holds for John domains $\Omega$ satisfying
$$ \mathcal H^{n-1}(\partial \Omega\setminus \partial_*\Omega)=0,$$ 
where $\partial_*\Omega$ denotes the measure-theoretic boundary, together with  an upper density bound on $\partial \Omega$. This class of domains includes  $(\ez,\,r)$-perimeter minimizers of Wulff perimeter $P_K$ which are close to the associated convex body $K$. Particularly, this result is established without requiring $\partial \Omega$ to be Ahlfors regular. 

As a consequence, we give an alternative proof for a crucial step in the quantitative Wulff inequality, thereby providing a meaningful commentary on the seminal work of Figalli, Maggi, and Pratelli \cite{FMP2010}. 
\end{abstract}


\maketitle
\section{Introduction}

 Let $K \subset \mathbb{R}^n, n\ge 2$ be a convex (open) set containing the origin such that $|K| = |B|$.
Here $B$ denotes the standard Euclidean unit ball and  $|\cdot|$ denotes the Lebesgue measure.  We denote by $\mathcal H^{k}$ the $k$-dimensional Hausdorff measure, and define
$$P_K(E)=\int_{\partial^*E} \|\nu_E(x)\|_{*}\, d\mathcal H^{n-1}(x)$$
where $\|\cdot\|_{*}: \mathbb{R}^{n} \to \mathbb{R}_+$ is defined as $$\|y\|_{*}:=\sup_{x\in K}x\cdot y  \ \text{ for any } y\in \mathbb R^n,$$
which is convex, positive, and  $1$-homogeneous. Here, $E$ represents a set of finite perimeter, $\partial^*E$ denotes its reduced boundary, and $\nu_E$ refers to the (measure-theoretic) unit outer normal; see the beginning of SubSection~\ref{sec:minimizer john} for more specific definitions. We further assume that 
$$\int_K x=0.$$
This assumption is not restrictive for our problem; see Remark \ref{affine invar} below for a detailed explanation.

We usually refer to 
 $P_K(E)$ as the \emph{Wulff perimeter} of $E$, and $K$ as the \emph{Wulff shape} corresponding to the surface tension 
 $\|\cdot\|_{*}$. 
The (open) set $K$ can be characterized via $\|\cdot\|_{*}$ by
$$K=\bigcap_{\nu\in \mathbb S^{n-1}} \left\{x\in \mathbb R^n\colon x\cdot \nu < \|\nu\|_{*} \right\}.$$
In particular, when $K=B$, or equivalently $\|\cdot\|_{*}\equiv 1$ on $\mathbb S^{n-1}$,  one obtains the standard Euclidean norm and denotes by 
$$P(E):=P_B(E)=\mathcal{H}^{n-1}(\partial^* E)$$
the perimeter of $E$ with respect to the Euclidean norm. 

Analogous to the standard Euclidean case, we have the Wulff inequality: For any set of finite perimeter  $E\subset\mathbb R^n$, 
$$P_K(E)\ge n |K|^{\frac {1} n }|E|^{\frac {n-1}n}.  $$
Moreover, the equality holds if and only if $E$ is congruent to $K$ up to translation and dilation.

In their distinguished work \cite{FMP2010}, Figalli, Maggi, and Pratelli employed a mass transportation approach to demonstrate a quantified version of the Wulff inequality:  For every set  of finite perimeter $E\subset\mathbb R^n$ with $|K|=|E|=|B|$, one has 
\begin{equation}\label{quantitative wulff}
P_K(E) - P_K(K)\ge c(n) \min_{x\in \mathbb R^n} |E\Delta (x+K)|^2,
\end{equation} 
where the constant is independent of $K$. 
Their result holds significant implications for exploring the geometric configurations of small liquid droplets and crystalline structures; see e.g. \cite{FM2011, FZ2022}.

\begin{rem}\label{affine invar}
   For any affine map $L:\mathbb R^n\to \mathbb R^n$, \eqref{quantitative wulff} is invariant under affine maps according to \cite[Step 2, Proof of Theorem 1.1]{FMP2010}. Hence, up to a translation, we may assume the barycenter of $K$ is the origin. 

   Moreover, recall  John's lemma  \cite[Theorem III]{J1948} yields that, for any convex set $K\subset\mathbb R^n$, there exists an affine map $L_0$ on $\mathbb R^n$ so that
  $$B_1\subset L_0(K)\subset B_n, \quad \det L_0>0. $$
Thus, we also assume that
\begin{equation}\label{mk control}
    1\le  \frac{M_K}{m_K}\le n.
\end{equation}
where $M_k$ and $m_k$ are defined as 
\begin{equation}\label{defi mk}
   m_K:=\inf\{\|v\|_*: v\in \mathbb S^{n-1}\}\quad \text{and} \quad M_K:=\sup\{\|v\|_*: v\in \mathbb S^{n-1}\}.
\end{equation}
\end{rem}

\medskip

An indispensable step in their proof,  presented in \cite[Theorem 3.4]{FMP2010}, requires identifying, for any set $E\subset \mathbb R^n$ close to the Wulff shape $K$ with small isoperimetric deficit
$$\delta(E):=\frac{P_K(E)}{n|K|^{\frac{1}{n}}|E|^{\frac{n-1}{n}}}-1\ll 1,$$
a corresponding set $F\subset \mathbb R^n$. This set $F$, while is close to $E$, supports either a Sobolev-Poincar\'e inequality of the form:
\begin{equation}\label{SP inequ}
\int_{E} \|-Du(x)\|_* \,dx \ge c(n) \inf_{a\in \mathbb R} \left(\int_{E} |u(x)-a|^{\frac n{n-1}}\, dx\right)^{\frac {n-1} {n}} \quad \text{ for all } \ u\in C_c^1(\mathbb R^n),
\end{equation}
wherein one  can derive a variant of \eqref{quantitative wulff} for $F$ with a non-optimal exponent of $4$ on the right-hand side; or a trace inequality of the form
\begin{equation}\label{1 trace inequ}
\int_{E} \|-Du(x)\|_* \,dx \ge c(n) \inf_{a\in \mathbb R} \int_{\partial E} |u(x)-a|\|\nu_E(x)\|_*\, d\mathcal H^{n-1}(x) \quad \text{ for all } \ u\in C_c^1(\mathbb R^n),
\end{equation}
wherein one can precisely establish the inequality \eqref{quantitative wulff} for $F$ with the optimal exponent of 2; see \cite[Section 1.6]{FMP2010} for more discussions.
We emphasize that the constants in \eqref{SP inequ} and \eqref{1 trace inequ} are independent of $K$ once the ratio $\frac{M_k}{m_k}$ is bounded uniformly. 
The  construction of $F$ in \cite{FMP2010} is well-executed, utilizing a 'maximal critical set' to derive an explicit constant in the trace inequality.

On the other hand, a sharp quantitative isoperimetric inequality for the standard Euclidean norm was proven by Fusco, Maggi, and Pratelli \cite{FMP2008} through the quantification of Steiner symmetrization. Later, Cicalese and Leonardi provided an alternative argument in \cite{CL2012} (see also \cite{AFM2013}), based on the regularity of minimal surfaces. In their approach, they also needed to slightly modify a set   $E\subset \mathbb R^n$ that is close to the unit ball $B$,  in order to obtain a set $F$ which is $W^{1,\,\infty}$-close to $B$.  To achieve this, they first established a penalized variational problem based on $E$ to obtain the set $F$, and then apply the regularity results for $(\ez,\,r)$-minimizers (see Definition~\ref{ez r} below) to show that $F$ is even $C^{1,\,\az}$-close to $B$ for some $0<\az<\frac 1 2$. 
This method is now known as the \emph{selection principle} and has found wide application; see, for example, \cite{BDV2015, FZ2017, FZ2022}.

However, a direct application of the regularity method, namely the selection principle, in the argument presented in \cite{FMP2010} encounters a challenge: the  Lipschitz regularity (uniformly under appropriate normalization) for $(\ez,r)$-minimizers with respect to $P_K$ still poses an unresolved problem when $n\ge 3$; we draw attention to \cite{N2016} for a related endeavor in this direction. Indeed, the shape $K$ itself may lack $C^{1,\,\alpha}$ regularity and uniform convexity, thereby complicating prospects for enhanced regularity.  To our knowledge, the most notable advancement in this regard can be found in \cite{AP1999} and \cite[Proposition 4.6]{ANP2002}, where a Lipschitz approximation property was demonstrated for $(\ez,r)$-minimizers of $P_K$ associated with a convex set $K$ of arbitrary form. One may consult the references therein for further literature on this topic.

\subsection{Literature: Sobolev-Poincar\'e meets John}
Nevertheless, Lipschitz regularity is not fully necessary to support \eqref{SP inequ} (or \eqref{1 trace inequ}). Indeed, there have been numerous studies on the geometric characterizations of domains that support the Sobolev-Poincar\'e inequality. In particular, it was shown in \cite{B1988} and \cite{GR1983} that a John domain admits a $(p^*,\,p)$-Sobolev-Poincar\'e inequality with the same possible exponents $p^*$ as the one for a ball, where $1\le p<n$ and $p^*=\frac{np}{n-p}$. Later, Buckely and  Koskela \cite{BK1995} also proved that the John property is necessary under some mild geometric assumptions on the domain. 
This relation was also generalized to  other metric measure spaces including Carnot groups; see e.g. \cite{HK2000} for a comprehensive study.

However, it appears to us that there is currently no direct discussion on the relationship between the trace inequality and John domains until now.
To be specific, let us introduce  the concept of John domains.
\begin{defn}
    For $J\ge 1$, a (bounded) domain $\Omega\subset \mathbb R^n$ is said to be  $J$-John  provided that, there exists a distinguished point $x_0\in \Omega$ so that, for any $x\in \Omega$, there exists a curve $\gamma\subset \Omega$ starting from $x$ ending at $x_0$ satisfying the following condition: 
\begin{equation}\label{John curve}
\ell(\gamma[x,\,y])\le J\dist(y,\,\partial \Omega) \quad \text{ for any $y\in \gamma$}, 
\end{equation}
where $\ell(\gamma[x,\,y])$ denotes the Euclidean length of the subcurve of $\gamma$ joining $x$ to $y$. 
We usually call $x_0$ the  {John center} of $\Omega$ and $\gamma$ the  {John curve} joining $x_0$ and $x$. 

In addition, a domain is $(J,\,s)$-John if there exists $s>0$ and $J\ge 1$ so that, whenever $z\in \partial E$ and $0<r<s$, for all points $x\in B_r(z)\cap \Omega$, there exists a point $w_x\in B_{Jr}(z)\cap \Omega$ with 
$$\dist(w_x,\,\partial \Omega)\ge J^{-1}r$$
and a curve $\gamma\subset \Omega$ joining $x$ to $w_x$ so that \eqref{John curve} is satisfied for any $y\in \gamma$.
\end{defn}

 It is worthy to note that $(J,s)$-John domain might not be John as the center $w_x$ depends on $x$. However, when the boundary of a $(J,s)$-John domain $\Omega$ is sufficiently close to that of a John domain in the Hausdorff distance, one can prove $\Omega$ is also a John domain, quantitatively; see Lemma \ref{John minimizer} below.

\subsection{Trace inequality} 

In order to attain the precise exponent  $2$, we establish the validity of \eqref{1 trace inequ} for functions in the space 
$$ W^{1,\,1}(\Omega):=\{u\in L^1(\Omega): |D u|\in L^1(\Omega)\}$$
within a John domain $\Omega$, subject to certain supplementary conditions.

For a Borel set $E$ and 
$x\in E$, we denote by $E^{(\lambda)}$ the set 
$$E^{(\lambda)}:=\left\{x\in\mathbb R^n: \lim_{r\to 0}\frac{|E\cap B_r(x)|}{|B_r(x)|}=\lambda\right\}.$$
Then the \emph{measure-theoretic boundary} $\partial_* E$ of $E$ is defined as $\partial_* E:=\mathbb R^n\setminus (E^{(0)}\cup E^{(1)})$.
Now we  illustrate  the following theorem, where such a theorem has beed studied under the Ahlfors--David regular assumption in \cite[Section 7]{DS1998}.
\begin{thm}\label{trace inequ}
Let  $\Omega\subset\mathbb R^n$ be a bounded $J$-John domain and of finite perimeter. Suppose that $\Omega$ satisfies
\begin{equation}\label{almost reduce}
    \mathcal H^{n-1}(\partial \Omega\setminus \partial_*\Omega)=0,
\end{equation} 
and there exists a constant $a_0>0$ so that, for  any $x\in \partial \Omega$ and any $0<r<r_0=c(n)\diam(\Omega),$   
\begin{equation}\label{hau cond 1}
    \mathcal H^{n-1}(\partial \Omega\cap B_{r}(x))\le a_0r^{n-1}. 
\end{equation}
Then for any $u\in W^{1,\,1}(\Omega)$ and $\mathcal H^{n-1}$-almost every $x\in \partial \Omega$, one has
\begin{equation}\label{trace defi}
    \lim_{r\to 0^{+}}\bint_{\Omega\cap B_r(x)}|u(y)-Tu(x)|dy=0
\end{equation}
exists, wherein the trace 
$$Tu(x):=\lim_{r\to 0^{+}}\bint_{\Omega\cap B_r(x)}u(y)dy$$ 
of $u$
is well-defined for $\mathcal H^{n-1}$-almost every   $x\in \partial \Omega$.

Moreover, for any $u\in W^{1,\,1} (\Omega)$, there exists a constant $C=C(n,\, a_0,\,J)$ such that 
\begin{equation}\label{trace conseq}
    \inf_{c\in \mathbb R}\int_{\partial \Omega} |Tu(x)-c|  \, d\mathcal H^{n-1}(x) \le C\int_{\Omega} |Du|\, dx.
\end{equation}
\end{thm}

\subsection{Application: A selection principle for the quantitative Wulff inequality}

Later, we apply Theorem~\ref{trace inequ} to $(\ez,\,r)$-minimizers, and conclude a selection principle for the quantitative Wulff inequality in \cite{FMP2010}. Let us recall the definition of $(\ez,\,r)$-minimizers. 
\begin{defn}\label{ez r}
   Given $\ez,\,r>0$, a set of finite perimeter $E\subset \mathbb R^n$ is an $(\ez,\,r)$-minimizer of $P_K$, if for any set $G\subset \mathbb R^n$ satisfying $E\Delta G\subset \subset x+rK$ with $x\in E$, one has
$$P_K(E)\le P_K(G)+ \ez |E\Delta G|.$$
\end{defn}

Observe that
an $(\ez,\,r)$-minimizer $E$ is a quasi-minimizer of the Euclidean perimeter $P$ in the sense of \cite[Definition 3.1]{KKLS2013}: for  any open set $U\subset\subset B_\rho(x)$ with  $\rho\in (0, r)$ and $x\in E$, and every set $G$ satisfying  $E\Delta G\subset \subset U$ 
, there exists $\kappa=\kappa(\ez,\,r,\,n)\ge 1$ so that 
\begin{equation}\label{quasiminimizer}
   P(E; U)\le\kappa P(G; U), 
\end{equation} 
where, for the distributional gradient $D\chi_E$ of the characteristic function $\chi_E$ of $E$, 
$$P(E;V)=\|D\chi_E\|(V)\quad \text{ for a measurable set } \ V\subset\mathbb R^n,$$
i.e. the total variation of $D\chi_E$ in $V$. We note that, in accordance with \cite[Theorem 9.6.4]{M2011}, the trace inequality \eqref{trace conseq} can also be derived under the condition that
$$\min\{P(G;\mathbb R^n\setminus \Omega),\,P(\Omega\setminus G;\mathbb R^n\setminus \Omega)\}\le C P(G;\,\Omega) \quad \text{ for any measurable set }\  G\subset \mathbb R^n. $$
Despite their similarities, this condition does not directly follow from \eqref{quasiminimizer} above.

Now the following theorem is a corollary of \cite{DS1998}; see also the discussion in \cite[Page 1620]{KKLS2013}. 
\begin{thm}[\cite{DS1998}]\label{main thm}
Let $\ez,\,r>0$, and $E\subset \mathbb R^n$ be an $(\ez,\,r)$-minimizer of $P_K$. Then there exists a constant $J_0=J_0\left(n,\epsilon, r\right)>0$ so that each component of $E$ is a  $(J_0,\, cr)$-John domain with $c=c\left (n,\epsilon, r\right)>0$.
\end{thm}

For an $(\ez,\,r)$-minimizer $E$  of $P_K$, we may regard it as a topologically open set, and equivalently integrate with respect to $\mathcal H^{n-1}$ on either its topological boundary or reduced boundary, according to \cite[Corollary 3.6]{ANP2002}.

Let us expound briefly on how the conditions in Theorem~\ref{trace inequ} are fulfilled by $(\ez,\,r)$-minimizers of $P_K$, thereby enabling the application of the selection principle. As one shall see in the proof of Proposition~\ref{selection}, $\ez$ can be chosen depending only on $n$ while $r$ is absolute.  
Then the standard density estimates \cite[Theorem 21.11]{M2012} yields \eqref{almost reduce} and \eqref{hau cond 1} with $a_0=a_0(n)$.
Recall from Remark~\ref{affine invar} that we may assume that $K$ satisfies
$$1\le \frac{M_K}{m_K}\le n.$$
Moreover, when the set $E$ is close to $K$  in the sense of \eqref{close to ball} below, Lemma~\ref{John minimizer} establishes the (global) $J$-John property of $E$.
Then Theorem~\ref{trace inequ} immediately implies that \eqref{1 trace inequ} also holds for $(\ez,\,r)$-minimizers of $P_K$ close to $K$.
 Consequently, we are empowered to utilize Theorem~\ref{trace inequ} and conclude the following proposition.

\begin{prop}\label{selection}
Let $1\le \frac{M_K}{m_K}\le n$, $\delta_k\to 0^+$ and $E_k\subset \mathbb R^n$ be a sequence of sets of finite perimeter for which  
\begin{equation}\label{close to ball}
(1-\delta)K\subset E_k\subset(1+\delta)K 
\end{equation}
for some $\delta=\delta(n)>0$ small, $|E_k|=|K|$, 
$E_k\to K$ in measure and
\begin{equation}\label{assumption}
    P_K(E_k)-P_K(K)\le \delta_k \min_{x\in\mathbb R^n}|E_k\Delta (x+K)|^2.
\end{equation}
We additionally assume that 
$\int_{E_k} x = 0.$

Then there exists a sequence of $J$-John domains $F_k\subset \mathbb R^n$ satisfying $|F_k|=|K|$ and  
\begin{equation}\label{barycond}
    \int_{F_k} x = 0.
\end{equation}
In addition,   
\begin{equation}\label{John selection}
    P_K(F_k)-P_K(K)\le \az_k \min_{x\in\mathbb R^n}|F_k\Delta (x+K)|^2
\end{equation}
with $J=J\left(n\right)$ and $\az_k\to 0$ as $k\to 0$. 
\end{prop}

Given that the proof of Proposition~\ref{selection} follows a standard approach and closely mirrors the one found in \cite[Pages 560-562]{F2015}, we have included it in the Appendix for completeness. To some extent, this proposition indicates that the stability of the (anisotropic) isoperimetric inequality derives from the stability of the John property of the minimizers under small perturbations.

The structure of our paper unfolds as follows:  In Section~\ref{sec:selection}, we firstly establish the trace inequality presented in Theorem~\ref{trace inequ}, and then confirm the selection principle outlined in Proposition~\ref{selection}, detailed in the Appendix.

{\bf Acknowledgement: } The authors wish to extend their heartfelt gratitude to the referee for their thoughtful review and constructive suggestions on an earlier draft of this manuscript. We are particularly grateful for the insightful observation that the result concerning the John regularity of $(\ez,\,r)$-minimal surfaces, as presented in the previous version, follows   from the foundational works \cite{DS1998} and \cite{KKLS2013}. We sincerely apologize for any oversight or inconvenience this may have caused and appreciate the opportunity to correct and improve the manuscript based on this feedback.

\section{Trace inequality in John domains}\label{sec:selection}

\subsection{Preliminary}\label{sec:minimizer john}

Let us establish some notation.
For a (rectifiable) curve $\gamma$, we denote by $\ell(\gamma)$ the Euclidean length of $\gamma$. When $\gamma$ is a curve, for any pair of points $x,\,y\in \gamma$, denote the subcurve joining $x$ to $y$ by $\gamma[x,\,y]$. For any measurable set $A\subset \mathbb R^n$, we define 
$$\bint_A u\,dx=\frac{1}{|A|}\int_A u\, dx\quad \text{for any }u\in L^1(A).$$

Recall that a measurable set $E\subseteq \R^n$  is said to have \emph{finite perimeter} if the distributional gradient of its characteristic function $\mathbf{\chi}_E$ is an $\mathbb R^n$-valued Radon measure $D\mathbf{\chi}_E$ with finite total variation, i.e., $|D\mathbf{\chi}_E|(\R^n)<\infty$.
According to the Lebesgue-Besicovitch differentiation theorem for measures, for $|D\mathbf{\chi}_E|$-a.e. $x$, it holds
\begin{equation}\label{cond}
\lim_{r\to 0^+} - \frac{D\mathbf{\chi}_E(x+rB^n)}{|D\mathbf{\chi}_E|(x+rB^n)} = \nu_E(x) \quad \text{and} \quad |\nu_E(x)| = 1.
\end{equation}
The set of points $x$ where this condition holds is called the \emph{reduced boundary} of $E$ and is denoted by $\partial^*E $.
 At points on the reduced boundary, $\nu_E(x)$ represents the \emph{measure-theoretic outward unit normal} to $E$ at $x$.
Also, up to changing $E$ in a set of measure zero, one can assume that $\overline{\partial^* E}=\partial E$.
We refer the interested reader to  \cite[Sections 12 and 15]{M2012} for more details on sets of finite perimeter.

\subsection{Trace inequality}

In this section, we always consider $K$ so that \eqref{mk control} is satisfied. 
We first record the following observation. 

\begin{lem}\label{lem:boundary curve}
    Let $\Omega\subset\mathbb R^n$ be a John domain with John center $x_0\in \Omega$. Then for any $x\in \partial \Omega$, there exists (at least) one  curve $\gamma_x\subset \overline{\Omega}$ with $\gz_x\setminus\{x\}\subset \Omega$ from $x$ to $x_0$ for which \eqref{John curve} holds. 
\end{lem}
\begin{proof}
     Take a sequence of points $x_i\in \Omega$ approaching $x$,  and the corresponding John curve $\gz_{x_i}\subset \Omega$ joining $x$ to $x_0$. Then $\ell(\gz_{x_i})$ is uniformly bounded according to \eqref{John curve}. Now by parametrizing $\gamma_{x_i}$ via arc length and up to relabeling the sequence, Arzel\'a-Ascoli lemma yields the uniform convergence of  $\gz_i$ to some curve $\gz_x\subset \overline{\Omega}$ joining $x$ to $x_0$. 
     
Moreover, the uniform convergence also implies that $\gz_x$  satisfies \eqref{John curve}. This gives 
$$\gz_x\setminus\{x\}\subset \Omega.$$ 
This concludes the lemma.  
\end{proof}

 Now we introduce the definition of an admissible domain. 

\begin{defn}\label{admissible}
    A bounded domain $\Omega\subset\mathbb R^n$ with finite perimeter is said to be admissible provided
    \begin{enumerate}
        \item [(i)] The measure-theoretic boundary of $\Omega$ almost coincides with its topological boundary, i.e.\
        \begin{equation}\label{almost theoretic}
            \mathcal{H}^{n-1}(\partial \Omega\setminus\partial_{*}\Omega)=0.
        \end{equation}
        \item [(ii)] For any $\widetilde x\in\partial\Omega$, there exists a positive constant $\Theta=\Theta(\Omega)$ and a ball $B_r(\widetilde x)$ so that 
         \begin{equation}\label{admissible 2}
         \mathcal H^{n-1}((\partial_{*} \Omega)\cap (\partial_* E))\le \Theta \mathcal H^{n-1}(\Omega\cap (\partial_* E))
         \end{equation}
         holds for each measurable set $E\subset\overline{\Omega}\cap B_r(\widetilde x)$.
    \end{enumerate}
\end{defn}

     In  \cite{Z1989}, it is shown that for an admissible domain $\Omega\subset \mathbb R^n$, the trace $Tu$ of the function  $u\in W^{1,1}(\Omega)$ can be defined in $L^1(\partial\Omega)$. To be more precise, as long as we prove that all the John domains $\Omega\subset\mathbb R^n$ satisfying the assumptions in Theorem~\ref{trace inequ} are admissible, \cite[Theorem 5.14.4]{Z1989} immediately ensures the existence of the trace for each $u\in W^{1,1}(\Omega)$. Additionally,   \cite[Theorem 5.10.7]{Z1989}   gives 
 $$\inf_c \int_{\partial\Omega} |Tu(x)-c|\, d\mathcal{H}^{n-1}(x)\le C(\Omega)\int_{\Omega} |Du(x)|\,dx.$$
  However, the inequality above is \emph{insufficient} for our purposes, even for the standard Euclidean norm, as the dependence of $C(\Omega)$ on $\Omega$ is not explicit in \cite[Theorem 5.10.7]{Z1989}. To address this issue, we present an alternative proof based on the John property, which provides an explicit dependence on the parameters.

Let us introduce the Whitney decomposition for an open set $\Omega\subsetneqq\mathbb R^n$.  For a  constant $c>0$ and any cube $Q\subset\mathbb R^n$  with center $x_Q=(a_1,\cdots, a_n)$ and sides parallel to the coordinate axis, we let $\ell(Q)$ be the edge length of $Q$ and then  rewrite 
$$\{x=(x_1,\cdots,x_n)\in \mathbb R^n: -c\ell(Q)/2\le x_i-a_i\le c\ell(Q)/2\}$$
as $cQ$ for brevity. Now we can state the following lemma on Whitney decomposition, see e.g.\ \cite[Chapter VI]{S1970}.

\begin{lem}[Whitney decomposition]\label{lma:whitney}
 For any open set $\Omega \neq \R^n$ there exists a collection
 $\mathscr F=\{Q_j\}_{j\in\mathbb N}$ of countably many closed dyadic cubes such that

 \begin{equation}\label{whitney cover v}
    \bigcup_{Q\in \mathscr{F}}Q=\Omega, \quad \chi_\Omega\le \sum_{Q\in\mathscr{F}}\chi_{\frac{11}{10}Q}\le C(n)\chi_\Omega.
\end{equation}

Moreover, for any $Q_i,\,Q_j\in\mathscr{F}$ with $Q_i\cap Q_j\neq \emptyset$, one has 
\begin{equation}\label{neig cube}
    \sqrt{n}\ell(Q_i)\le \dist(Q_i,\partial\Omega)\le 4\sqrt{n}\ell(Q_i)\quad\text{and}\quad \frac{1}{4}\le \frac{\ell(Q_i)}{\ell(Q_j)}\le 4.
\end{equation}
\end{lem}

 We first show that every John domain in Theorem~\ref{trace inequ} is admissible. 

 \begin{lem}\label{admissible proof}
     A bounded $J$-John domain $\Omega\subset\mathbb R^n$ with finite perimeter is admissible, provided \eqref{almost reduce} and \eqref{hau cond 1} are satisfied. Then every function $u\in W^{1,1}(\Omega)$ has a trace in  $L^1(\partial \Omega)$.
 \end{lem}
\begin{proof} 
We firstly note that \eqref{almost theoretic} immediately follows from the assumption \eqref{almost reduce}.

Suppose that $x_0\in \Omega$ is the John center of $\Omega$ and $\mathscr{F}$ is the Whitney decomposition of $\Omega$. 
Next we verify \eqref{admissible 2}. Toward this,  
for any $\widetilde x\in \partial \Omega$, we first choose  $r>0$ with $x_0\notin B_{3r}(\widetilde x)$, and let $E\subset\overline{\Omega}\cap B_r(\widetilde x)$ be a measurable set.

According to \eqref{admissible 2}, we may assume that 
$$\mathcal H^{n-1}(\Omega\cap \partial_* E)<\infty.$$
Since  $\Omega$ is of finite perimeter, Federer's theorem \cite[Theorem 16.2]{M2012} together with the fact
$$\partial_*E = (\partial_* E \cap \partial \Omega)\cup (\Omega\cap \partial_* E) $$
implies that $E$ also is of finite perimeter.
Consequently,  we may further assume that $E$ has finite perimeter. 

\medskip
\noindent{\bf Step 1: Replace $E$ by a smooth open set.} We first show the existence of   a smooth open set $\widetilde E\subset \mathbb R^n$ satisfying 
\begin{equation}\label{replace}
    \mathcal H^{n-1}(\partial \widetilde E\cap \partial \Omega)=0,\  \mathcal{H}^{n-1}(\partial_*\Omega\cap \partial_*E)\le 2\mathcal{H}^{n-1}(\partial_*\Omega\cap \widetilde E)\ \text{and}\  P(\widetilde E;\overline{\Omega})\le 2\mathcal{H}^{n-1}(\partial_*E\cap \Omega).
\end{equation}

Toward this, let $\ez>0$, $\eta\in C^\infty_c(B)$ radially symmetric with $\int_{B_1} \eta\, dx=1$, and 
$$\eta_\ez(x)=\ez^{-n}\eta\left(\frac x \ez\right)\ge 0.$$
Define
$$u_{\epsilon}(x):=\int_{\mathbb R^n}\chi_E(y)\eta_\epsilon(x-y)dy =\int_{\mathbb R^n}\chi_{\frac{E-x}{\epsilon}}(z)\eta(z)dz.$$
Then when $x\in \partial^*E$, since $(E-x)/\epsilon$ converges to some hyperplane $H_x^-$ as $\epsilon\to 0$, we have 
\begin{align*}
    u_{\epsilon}(x)\to \int_{\mathbb R^n}\chi_{H_x^-}(y)\eta(z)dz=\frac{1}{2}\quad \text{as } \epsilon\to 0.
\end{align*}
Thus
\begin{equation}\label{open inclu}
    \partial^*E\subset \bigcup_{a>0}\bigcap_{0<\epsilon<a}\{x\in \mathbb R^n: u_{\epsilon}(x)>t\}
\end{equation}
holds for any $t\in (0,1/4)$. 

Choose a sequence $\{\epsilon_h\}_{h\in\mathbb N}$ satisfying $\epsilon_h \to 0^+$ and let $$u_h:=u_{\epsilon_h},\quad E_{h}^t:=\{x\in \mathbb R^n: u_{h}(x)>t\}.$$ 
Since $\Omega$ is John, by \cite[Theorem 2.8]{V2000},
$|\partial\Omega|=0$, and hence the coarea formula  yields, for any $h\in \mathbb N$, 
$$0=\int_{\partial\Omega}|Du_{h}|dx=\int_{\mathbb R} P(E_h^t;\partial\Omega) dt,$$ 
thus giving
\begin{equation}\label{perime0}
    P(E_h^t;\partial\Omega)=0\quad \text{for a.e. }t\in (0,1).
\end{equation}
Moreover, the argument of \cite[Theorem 13.8]{M2012} implies that
\begin{equation}\label{peri converge}
    P(E;\Omega)=\liminf_{h\to +\infty}P(E_h^t;\Omega) \quad \text{for a.e. } t\in (0,1).
\end{equation}
Since we only have countably many $h\in \mathbb N$ in question, there exists $t\in (0,\frac{1}{4})$ satisfying \eqref{perime0} and \eqref{peri converge} for any $h\in \mathbb N$.

Fix such a constant $t\in (0,\frac{1}{4})$. 
Then thanks to \eqref{open inclu} and Federer's theorem, for $h\gg 1$. we obtain that  
$$\mathcal{H}^{n-1}(\partial_*\Omega\cap \partial_*E)=\mathcal{H}^{n-1}(\partial_*\Omega\cap \partial^*E)\le 2\mathcal{H}^{n-1}(\partial_*\Omega\cap E_h^t)$$
and from \eqref{perime0} and \eqref{peri converge} that
$$
 \mathcal H^{n-1}(\partial E^t_h\cap \partial \Omega)=0 \quad \text{ and } \quad  P(E_h^t;\overline{\Omega})= P(E_h^t;\Omega)\le 2P(E;\Omega)= 2\mathcal{H}^{n-1}(\partial_*E\cap \Omega).
$$
 Thus, $\widetilde E:=E_h^t$ is the desired open set satisfying \eqref{replace}.
 
 Recall that $E\subset \overline{\Omega} \cap B_r(\widetilde x)$. As $h\gg 1$, we also have   
$\widetilde E\subset B_{2r}(\widetilde x)$.
 Additionally, in order to prove \eqref{admissible 2}, \eqref{replace} tells that it suffices to show 
\begin{equation}\label{regu epsi}
\mathcal{H}^{n-1}(\partial_*\Omega\cap \widetilde E) \le \frac{\Theta }{4}P(\widetilde E;\overline{\Omega}).
\end{equation}
\medskip
\noindent{\bf Step 2: A regularization of $\widetilde E$ .}
Next, we aim to replace $\widetilde E$ by a more regular set so that it has certain lower density estimate in $\Omega$.

To this end, we consider a set of finite perimeter $ U\subset \mathbb R^n$ minimizing the Plateau problem:
$$\inf\{P(G): G\setminus \Omega = \widetilde E\setminus\Omega\}.$$
From \cite[Proposition 12.29]{M2012} and \cite[Example 16.13]{M2012},  it follows that $ U$ exists and is a perimeter minimizer. Additionally, we also have $U\subset B_{2r}(\widetilde x)$. Otherwise, since
$$P(U\cap B_{2r}(\widetilde x))+P(U\cup B_{2r}(\widetilde x))=P(U) + P(B_{2r}(\widetilde x)),$$
the isoperimetric inequality yields
$$P(U\cap B_{2r}(\widetilde x))-P(U)= P(B_{2r}(\widetilde x))-P(U\cup B_{2r}(\widetilde x))< 0,$$
and then, by recalling that $U\setminus \Omega=\widetilde E\setminus \Omega\subset B_{2r}(\widetilde x)$, $U$ cannot be a minimizer.

Moreover, 
\begin{equation}\label{desity estimate mini}
c(n)\rho^{n-1}\le P(U;B_\rho(y))\le n|B|\rho^{n-1} \quad \text{ for any  $B_{2\rho}(y)\subset\subset \Omega$ with $y\in \Omega\cap \partial U$},
\end{equation}
 given by \cite[Theorem 16.14]{M2012}.  
In particular,  combining \eqref{neig cube} and \eqref{desity estimate mini},  for any Whitney cube $Q\in \mathscr{F}$ with $Q\cap \partial U\neq \emptyset$, there exists a constant $c_1=c_1(n)>0$, such that  
 \begin{equation}\label{lower density min}
     \mathcal{H}^{n-1}\left(\partial  U\cap \frac{11}{10} Q\right)\ge c_1(n)l(Q)^{n-1}.
\end{equation}

In addition, by \cite[Exercise 12.16]{M2012}, the constraint of $U$ outside $\Omega$ yields $$P(U; \mathbb R^n\setminus \overline{\Omega})=P(\widetilde E; \mathbb R^n\setminus \overline{\Omega}).$$
This, combined with the minimality of $U$, 
implies
\begin{equation}\label{rh2}
    P(U;\overline{\Omega})\le P(\widetilde E;\overline{\Omega}). 
\end{equation} 
Moreover, due to the openness of $\widetilde E$, the constraint of $U$ outside $\Omega$ also gives 
\begin{equation}\label{lh2}
     (U\setminus\Omega)\cap B_r(x) =(\widetilde E\setminus\Omega)\cap B_r(x)=B_r(x)\setminus\Omega\quad \text{for any }x\in \partial_*\Omega\cap \widetilde E\text{ and }r\ll 1.
\end{equation}
Hence, it follows from \eqref{lh2} that $U^{(0)}\cap \partial_*\Omega\cap U =\emptyset$. Indeed, for every $z\in  \partial_*\Omega\cap U$, then \eqref{lh2} implies, for $r\ll1$,
\begin{align*}
    |U\cap B_r(z)|=& |(U\setminus\Omega)\cap B_r(z)|+|(U\cap\Omega)\cap B_r(z)|\ge |(U\setminus\Omega)\cap B_r(z)|=|B_r(z)\setminus \Omega|, 
\end{align*} 
and hence $z\notin U^{(0)}$ as  $z\in\partial_*\Omega.$
Thus via Federer's theorem again, 
\begin{equation}\label{lh3}
    \mathcal{H}^{n-1}(\partial_*\Omega\cap \widetilde E)=\mathcal{H}^{n-1}(\partial_*\Omega\cap U)\le \mathcal{H}^{n-1}(\partial_*\Omega\cap U^{(1)})+\mathcal{H}^{n-1}(\partial_*\Omega\cap \partial^* U).
\end{equation}
Besides, we have
\begin{equation}\label{rh3}
    P(U;\overline{\Omega})=\mathcal{H}^{n-1}(\partial U\cap \Omega)+\mathcal{H}^{n-1}(\partial^*U\cap \partial\Omega).
\end{equation}

We claim that, for some  constant $\Theta'=\Theta'(n,\,J,\,a_0)>0$
\begin{equation}\label{E to U}
    \mathcal H^{n-1}( U^{(1)}\cap  \partial \Omega)\le \Theta' \mathcal H^{n-1}(\partial U\cap  \Omega).
\end{equation} 
Once this holds, then by  \eqref{rh2}, \eqref{lh3} and \eqref{rh3}, 
\begin{align*}
    \mathcal{H}^{n-1}(\partial_*\Omega\cap \widetilde E)=& \mathcal{H}^{n-1}(\partial_*\Omega\cap U)\le  \mathcal{H}^{n-1}(\partial_*\Omega\cap U^{(1)})+\mathcal{H}^{n-1}(\partial_*\Omega\cap \partial^* U)\\
    \le & \mathcal{H}^{n-1}(\partial\Omega\cap U^{(1)})+\mathcal{H}^{n-1}(\partial\Omega\cap \partial^* U)\\
 \le & \Theta'\mathcal H^{n-1}(\partial U\cap  \Omega) +   P(U;\overline{\Omega})-\mathcal{H}^{n-1}(\partial U\cap \Omega)\\
 \le & (1+\Theta') P(U;\overline{\Omega})\le (1+\Theta')P(\widetilde E;\overline{\Omega}),
\end{align*}
which immediately yields \eqref{regu epsi} with $\Theta=4(1+\Theta')$.

\medskip 
\noindent{\bf Step 3: Proof of \eqref{E to U}.}  
Recall that $\Omega$ is a John domain with center $x_0\in \Omega$.  Then for any $x\in U^{(1)}\cap \partial \Omega$, we first claim that, there exists a $C(J)$-John curve $\widetilde\gz_x$ with $\tilde \gz_x\setminus\{x\}\subset \Omega$ such that
\begin{align}\label{claim:curve}
  \widetilde \gz_x\setminus\{x\}\text{ intersects $\partial U\cap \Omega$ at some point $y_x\in \partial U\cap \Omega$.}
\end{align}

Indeed, suppose that the John curve $\gamma_x$ given by Lemma~\ref{lem:boundary curve} 
intersects $\overline{U}\cap \Omega$. Then by the connectivity of $\gamma_x$ and the fact that
$$x_0\in \mathbb R^n \setminus B_{3r}(x)\subset \mathbb R^n\setminus U,$$ we conclude 
$(\gamma_x\setminus\{x\})\cap\partial(U\cap\Omega)\neq \emptyset.$
Since $\gamma_x\setminus\{x\}\subset \Omega$, then it follows that
$$(\gamma_x\setminus\{x\})\cap\partial(U\cap\Omega)\subset \partial U\cap \Omega,$$
and the claim follows. Thus, we may assume that   $\gamma
_x$ satisfies
$$\gamma_x\setminus \{x\}\subset (\mathbb R^n\setminus \overline{U})\cap \Omega.$$

Since the definition of John curve yields that, for any $z\in \gamma_x$, 
$$J\dist(z,\,\partial \Omega)\ge  \ell(\gamma_x[x,\,z]), $$
then we obtain that 
$$B(z,\ell(\gamma_x[x,z])/(3J))\subset\subset\Omega. $$
Since $x\in U^{(1)}$, when $z$ is sufficiently close to $x$, one has 
\begin{equation}\label{bz intersection}
    B(z,\ell(\gamma_x[x,z])/(3J))\cap U\cap \Omega\neq \emptyset.
\end{equation} 
Otherwise, for 
$$s:=\left(1+\frac{1}{3J}\right)\ell(\gamma_x[x,z]) \ge \left(1+\frac{1}{3J}\right) |x-z|,$$
one has
$$B(z,\ell(\gamma_x[x,z])/(3J))\subset B(x,\,s)\setminus U,$$
and hence
\begin{align*}
    \frac{|U\cap B(x,s)|}{|B(x,s)|}\le 1-\frac{|B(z,\, \ell(\gamma_x[x,z]))/(3J))|}{|B(x,s)|}\le 1-\frac{|B(z,\,cs)|}{|B(x,s)|}\le 1-c^n,
\end{align*}
for some $0<c=c(J)<1$, a contradiction to $x\in U^{(1)}$. This proves \eqref{bz intersection}, and we take 
$$w_z\in B(z,\ell(\gamma_x[x,z])/(3J))\cap (\partial U\cap \Omega).$$

Now according to \eqref{bz intersection}, by concatenating the curves 
$$\gamma_x[x,\,z] \cup [z,\,w_z]\cup [w_z,\,z]\cup \gamma_x{}[z,\,x_0]:=\tilde\gamma_x,$$
where we denote by $[z,\,w_z]$ the segment joining $z$ to $w_z$ (with the given direction), one gets the desired John curve $\tilde \gamma_x$ satisfying \eqref{claim:curve} since
$$\ell(\gamma_x[x,\,z]) + \ell([z,\,w_z]) + \ell([z,\,w_z])\le  \left(J+\frac 2 3\right) \dist(z,\,\partial\Omega)\le 2\left(J+\frac 2 3\right)\dist([z,\,w_z],\,\partial \Omega). $$

As a consequence of \eqref{claim:curve} and the definition of John curve \eqref{John curve}, we conclude that, for any $x\in U^{(1)}\cap \partial \Omega$, one has
$$|x-y_x|\le \ell(\tilde\gamma_x[x,\,y_x])\le J \dist(y_x,\,\partial\Omega).$$
Thus, according to \eqref{neig cube}, there exists $C_0=C_0(n,\,J)>0$ for which
\begin{equation}\label{Q cover boundary}
    x\in C_0 Q_x \quad \text{ whenever } \  x\in U^{(1)}\cap \partial \Omega \ \text{ and } \  y_x\in Q_x\in \mathscr F. 
\end{equation}

Now we are ready to show \eqref{E to U}. Let 
$$\mathscr S:=\{Q\in \mathscr{F}:Q=Q_x \text{ given by \eqref{Q cover boundary} for some $x\in  U^{(1)}\cap \partial \Omega$}\}.$$ 
Then it follows from \eqref{Q cover boundary} that 
\begin{equation}\label{boundary covering0}
      U^{(1)}\cap \partial\Omega\subset   \bigcup_{Q\in \mathscr S} C_0 Q.
\end{equation}
Furthermore, for any $Q\subset \mathscr{S}$, we choose $x_Q\in U^{(1)}\cap \partial\Omega$ and a ball $D_{Q}$ with center $x_Q$ such that 
\begin{equation}\label{ball covering}
    C_0 Q\subset D_{Q} \quad \text{and}\quad \diam(D_Q)=2 \sqrt{n}C_0\ell(Q).
\end{equation}
Hence, combining \eqref{boundary covering0} and \eqref{ball covering}, we have 
\begin{equation}\label{boundary covering}
      U^{(1)}\cap \partial\Omega\subset \bigcup_{Q\in \mathscr S} D_Q.
\end{equation}

Now since $\{Q\}_{Q\in \mathscr F}$ covers the whole domain $\Omega$, and $\{\frac {11}{10} Q\}_{Q\in \mathscr F}$ has at most $C(n)$-overlaps by \eqref{whitney cover v},  we conclude 
\begin{align}
       \sum_{Q\in \mathscr S}\mathcal{H}^{n-1}\left(\partial U\cap \frac{11}{10}Q\right)&\le    \sum_{Q\in \mathscr{F}}\mathcal{H}^{n-1}\left(\partial U\cap \frac{11}{10}Q\right) \nonumber \\
   &\le C(n) \sum_{Q\in \mathscr{F}}\mathcal{H}^{n-1}(\partial U\cap  Q)=C(n)  \mathcal{H}^{n-1}(\partial U\cap \Omega).\label{interior part}
\end{align}
On the other hand, \eqref{lower density min}
gives 
\begin{equation}\label{lower bound S}
  \sum_{Q\in \mathscr S}\mathcal{H}^{n-1}\left(\partial U\cap \frac{11}{10}Q\right)\ge C(n)\sum_{Q\in \mathscr S }l(Q)^{n-1}.   
\end{equation}
In addition, combining \eqref{hau cond 1}, \eqref{ball covering} and  \eqref{boundary covering}, we have
\begin{align}
    \sum_{Q\in \mathscr S }l(Q)^{n-1}& \ge C(n,\,J)\sum_{Q\in \mathscr S }\diam(D_Q)^{n-1} \ge C(n,\,J,\,a_0)\sum_{Q\in \mathscr S } \mathcal{H}^{n-1}\big(\partial\Omega \cap D_Q\big)\nonumber \\
    &\ge   C(n,\,J,\,a_0)\mathcal{H}^{n-1}(U^{(1)}\cap \partial \Omega).\label{boundary part}
\end{align}

Combining \eqref{interior part}, \eqref{lower bound S} and \eqref{boundary part}, we conclude \eqref{E to U}, and hence \eqref{admissible 2} for $E$. Therefore, $\Omega$ is admissible, and \cite[Theorem 5.14.4]{Z1989} gives the rest of the lemma. 
 The proof is completed.
\end{proof}

Now we are ready to prove Theorem~\ref{trace inequ}.
\begin{proof}[Proof of Theorem \ref{trace inequ}]

Let $x_0\in \Omega$ be the John center of $\Omega$ and $\mathscr{F}$ be the set of all Whitney cubes of $\Omega$. We choose $Q_0$ to be a cube in $\mathscr{F}$ with the John center $x_0\in Q_0$. For any $Q\in \mathscr{F}$, we denote by
$$\hat{Q}:=\frac{11}{10}Q,\qquad  u_{\hat{Q}}:=\bint_{\hat{Q}} u(x)dx.$$ 
Since $\Omega$ is admissible by Lemma~\ref{admissible proof}, then 
\eqref{trace defi} follows from  \cite[Theorem 5.14.4]{Z1989}.  We next show \eqref{trace conseq}. 

\medskip

\noindent{\bf Step 1 : Estimate $|u_{\hat{Q
}_i}-u_{\hat{Q
}_j}|$ for any pair of cubes  $Q_i,Q_j\in \mathscr{F}$ with $Q_i\cap Q_j\neq\emptyset$.} 
Note that there exists a cube $R$  with 
\begin{equation}\label{cap cube size}
    R\subset \hat{Q}_i\cap \hat{Q}_j\quad \text{and} \quad \ell(R)=\frac{1}{20}\min\{\ell(Q_i),\ell(Q_j)\}. 
\end{equation}
Then by the triangle inequality,  
$$|u_{\hat{Q
}_i}-u_{\hat{Q
}_j}|\le |u_{\hat{Q}_i}-u_R|+|u_R-u_{\hat{Q}_j}|.$$

As \eqref{neig cube} together with \eqref{cap cube size} gives $\ell(R)\ge \frac{1}{80}\ell(Q_i)$,  we apply the triangle inequality and  the $1$-Poincar\'e inequality on $\hat{Q}_i$ to conclude 
\begin{align}\label{poin 2}
    |u_{\hat{Q}_i}-u_R|&\le \bint_{\hat{Q}_i}\left(\bint_{R}|u(y)-u(z)|dz\right)dy\le C(n)\bint_{\hat{Q}_i}\left(\bint_{\hat{Q}_i}|u(y)-u(z)|dz\right)dy\nonumber\\
    & \le C(n)\bint_{\hat{Q}_i}|u(y)-u_{\hat{Q}_i}|dy \le C(n)\ell(Q_i)\bint_{\hat{Q}_i}|Du(y)|dy.  
\end{align}
Likewise, we obtain a similar upper bound for $|u_R-u_{\hat{Q}_j}|$, thus
\begin{equation}\label{estimate cube 6}
    |u_{\hat{Q}_i}-u_{\hat{Q}_j}|\le  C(n)\left(\ell(Q_i)\bint_{\hat{Q}_i}|Du(y)|dy+\ell(Q_j)\bint_{\hat{Q}_j}|Du(y)|dy\right).
\end{equation}

\noindent{\bf Step 1.2:  Estimate  $|Tu(x)-u_{\hat{Q}_0}|$ for any $x\in \partial \Omega$.} 
Fix $x\in \partial \Omega$. 
According to Lemma~\ref{lem:boundary curve},  there exists a $J$-John curve $\gamma_x$ joining $x$ to $x_0$. Moreover, the definition \eqref{John curve} of the John curve yields that, every Whitney cube $Q\in \mathscr{A}_x:=\{Q\in\mathscr{F}:Q\cap \gamma_x\neq\emptyset\}$  satisfies 
\begin{equation}\label{John cube}
    \hat{Q}\subset B_{C_1\ell(Q)}(x)\cap \Omega \quad \text{ and } \quad   x\in C_2Q,
\end{equation}
for some positive constants  $C_1=C_1(n,\,J)$ and $C_2=C_2(n,\,J)$, since the length of the subcurve of $\gamma_x$  between $x$ and $Q$ is at most $C(n,\,J)\diam(Q)$. 

Relabel $\mathscr{A}_x=\{Q_k\}_{k\in\mathbb N}$ so that $Q_k\cap Q_{k+1}\neq \emptyset, k\in \mathbb N$ together with $Q_k\to x$ as $k\to \infty$, and recall that $x_0\in Q_0$.
Then \eqref{trace defi} implies
\begin{equation}\label{convergence along curve 2}
   \lim_{k\to \infty}|u_{\hat{Q}_k}-Tu(x)| =0.
\end{equation}
Consequently, via \eqref{estimate cube 6}, we get that 
\begin{equation}\label{S2 estimate}
    |Tu(x)-u_{\hat{Q}_0}|\le \sum_{k=0}^{+\infty}|u_{\hat{Q
}_{k+1}}-u_{\hat{Q
}_k}|\le \sum_{k=0}^{+\infty} C(n)\ell(Q_k)\bint_{\hat{Q}_k}|Du(y)|dy.
\end{equation}

\noindent{\bf Step 1.3: Final estimate.} 
Now we integrate both sides of \eqref{S2 estimate} on $\partial \Omega$ with respect to $\mathcal H^{n-1}$-measure and obtain that
\begin{equation}\label{integral n-1}
\int_{\partial \Omega} |Tu(x)-u_{\hat{Q}_0}|\, d\mathcal H^{n-1}(x)\le C(n) \int_{\partial \Omega} \sum_{Q_k \in \mathscr A_x} \ell(Q_k)^{1-n}\int_{\hat{Q}_k}|Du(y)|\,dy\, d\mathcal H^{n-1}(x).
\end{equation}

By \eqref{John cube}, for each $Q\in\mathscr{F}$, choose a point $x_Q\subset C_2 Q\cap \partial  \Omega$ and then 
$$
    C_2Q \subset B_{C_1C_2\ell(Q)}(x_Q)=:D_Q.
$$
Hence, by \eqref{John cube} and the definition of $\mathscr{A}_x$, we have
$$\{x\in \partial \Omega\colon Q\in \mathscr A_x\}\subset D_Q\cap \partial \Omega.$$
Then since
$$\mathcal H^{n-1}(D_Q\cap  \partial \Omega)\le C(n,\,J,\,a_0) \diam(D_Q)^{n-1} \le C(n,\,J,\,a_0) \ell(Q)^{n-1}$$
according to \eqref{hau cond 1}, by interchanging the integral and summation on the right-hand side of \eqref{integral n-1} via Fubini's theorem, we arrive at
\begin{align*}\int_{\partial \Omega} |Tu(x)-u_{\hat{Q}_0}|\, d\mathcal H^{n-1}(x)\le & C(n,\,J,\,a_0)   \sum_{Q\in\mathscr F} \ell(Q)^{1-n}\mathcal H^{n-1}(D_Q\cap  \partial \Omega) \int_{\hat{Q}}|Du(y)|\,dy \\
\le & C(n,\,J,\,a_0) \sum_{Q\in\mathscr F}   \int_{\hat{Q}}|Du(y)|\,dy \le  C(n,\,J) \int_{\Omega} |Du(y)|\, dy.
\end{align*}
Thus we conclude \eqref{trace conseq}.
\end{proof}

\section{An extra step to Proposition \ref{selection}}

In this section, we always consider $K$ satisfying \eqref{mk control}. 
Towards Proposition \ref{selection}, we need the following hypothesis,  that the set $E\subset\mathbb R^n$ satisfies \eqref{close to ball}
for some $\delta=\delta(n)>0$. Indeed, via selection principle, towards the stability of Wulff inequalities with respect to $P_K$, it suffices to consider the case where \eqref{close to ball} is satisfied.

When $\delta>0$ is small enough, we show that any $(\ez,\,r)$-minimizer  $E$  of $P_K$is a John domain based on Theorem~\ref{main thm}. In particular, $E$ is connected. 

\begin{lem}\label{John minimizer}

Let  $0<\delta\le \frac{c}{3 nJ_0} r$, where $J_0$ and $c$ are the constants in  Theorem~\ref{main thm}. Then any (component of) $(\ez,\,r)$-minimizer $E$ satisfying \eqref{close to ball} is a $J$-John domain with $J=J(n, \epsilon, r).$
\end{lem}
\begin{rem}
Since $1\le \frac{M_K}{m_K}\le n$ by the additional assumption \eqref{mk control}, one has $c=c(n, \epsilon, r)>0$. Moreover, in the proof of Proposition~\ref{selection}  one chooses $r>0$ to be absolute, and then  $\delta$ depends only on $n$ eventually. 
\end{rem}

\begin{proof}[Proof of Lemma \ref{John minimizer}]
We show that every point $z\in E$ can be joined to the John center  $0$  of $E$,  which is also the origin, by a John curve $\gamma_z\subset E$ so that for any $a \in \gamma_z[z,0]$,  
\begin{equation}\label{John consequence}
    \ell(\gamma_z[z,\,a])\le J\dist(a,\,\partial E)
\end{equation}
holds for some $J=J(n, \epsilon, r)$. 

Towards this, for given $z\in E$,  let $y\in\partial K$ be the vector so that  $z=\lambda y$ for some $\lambda>0$.
Choose the outermost point $z_0\in \partial E$ which is in the same direction of $y$, i.e.\ $z_0:=\lambda_0 y$, and which is in the same component of $E$ as $z$. Then $\lambda_0>\lambda$. 

To achieve our proof, we discuss in two cases.

\noindent{\bf Case 1: $z\in (1-2\delta)K$.}
  Observe that $z_0$, $z$ and $0$ are on the same line. Thus, the convexity of $(1-2\delta)K$ and the assumption \eqref{close to ball} ensures that the line segment $L_z$ joining $z$ to $0$ satisfies $L_z\subset E$. Furthermore, recalling that  $z\in (1-2\delta)K$ and $y\in \partial K$ and that $z=\lambda y$ for some $\lambda>0$, for any $a\in L_z$,
  \begin{equation}\label{case1 est1}
      a=(|a|/|y|)y\quad \text{and} \quad |a|\le |z|< (1-2\delta )|y|.
  \end{equation}
Thus, via $y\in \partial K$, 
  the convexity of $(1-2\delta)K$ and the assumption \eqref{close to ball} also ensure that each point $a\in L_z$ satisfies 
  $$a+(1-2\delta-|a|/|y|)K\subset (1-2\delta)K\subset E.$$
  In particular, $\dist(a, \partial E) \ge (1-2\delta-|a|/|y|)\dist(0, \partial K)$.
Therefore, since \eqref{defi mk} gives
\begin{equation}\label{lem Kbd}
 \mathbb R^n\setminus  B_{\rho}\subset \partial K\subset B_{n\rho}\quad \text{for some }\rho\in(0,1)
\end{equation}
and since $y\in \partial K$,
we obtain from \eqref{case1 est1} that
\begin{align*}
    \dist(a, \partial E)&\ge (1-2\delta-|a|/|y|)\dist(0, \partial K)\ge ((1-2\delta)|y|-|a|)/n\nonumber\\
    & \ge (|z|-|a|)/n= \ell(L_z[z,a])/n.
\end{align*}
As a consequence, $L_z$ is the desired John curve with John constant $n$.


\noindent{\bf Case 2: $z\in E\setminus (1-2\delta)K$.} 
We aim to find a curve $\beta_z$ joining $z$ to a point  $z_1\in (1-2\delta)K$. Since we can join $z_1$ to the origin via a line segment according to Case 1, the desired John curve is obtained by concatenating these two curves.

To this end, recall that $z=\lambda y$ and $z_0=\lambda_0 y$. 
 \eqref{close to ball} implies $\lambda\in [1-2\delta,1+\delta)$ and  $\lambda_0\in (\lambda, 1+\delta)$ so that by \eqref{lem Kbd} and $y\in \partial K$, 
$$|z_0-z|= (\lambda_0-\lambda)|y|< 3\delta|y|\le 3n\delta.$$ 
Thus $z\in B_{3n\delta}(z_0)$ with $3n \delta\le J_0^{-1}cr$.

On the other hand, as Theorem \ref{main thm}, together with \eqref{mk control},  implies  $E$  is a $(J_0,\ c r)$-John domain for  $J_0=J_0(n, \epsilon, r)$, 
 we can  join $z$ to some point $z_1\in E\setminus B_{J_0cr}(z_0)$
by a curve $\beta_z\subset E$ satisfying
\begin{equation}\label{curve 1}
  \dist(z_1,\partial E)\ge J_0^{-1}cr \  \text{ and } \  \ell(\beta_z[z,\,a])\le J_0\dist(a,\,\partial \Omega) \quad  \text{for any } a\in \beta_z[z,\,z_1].
\end{equation}
Since $ J_0^{-1}cr\ge 3n\delta$, and $\partial E\subset (1+\delta) K\setminus (1-\delta)K,$ we conclude from \eqref{lem Kbd} and \eqref{curve 1} that
$z_1\in (1-2\delta)K$. Now by joining $z_1$ to the origin via a 
line segment $L_{z_1} \subset E$ via Case 1, we further get that
\begin{equation}\label{curve 2}
    \ell(L_{z_1}[z_1,\,a])\le n\dist(a,\,\partial E) \qquad  \text{for any } a\in L_{z_1}[z_1,\,0].
\end{equation}

Set $\gamma_z:=\beta_z\cup L_{z_1}$, which is a curve joining $z$ to $0$. Now we show that $\gamma_z$ is the desired John curve. 
By \eqref{curve 1}, it suffices to check points $b\in L_{z_1}.$


Since $b\in L_{z_1}$ with $L_{z_1}$ the segment joining $z_1$ to $0$, by the triangle inequality, the assumption that $\partial E\subset (1+\delta) K\setminus (1-\delta)K,$ together with the fact that $ \dist(z_1,\,\partial E) \ge  3n\delta$, 
\begin{align*}
    \frac 1 3 \dist(z_1,\,\partial E)\le  & \dist(z_1,\,\partial E)-2 n\delta \le  \dist(z_1,\,\partial((1-\delta)K)\\
    \le & \dist(b,\,\partial((1-\delta)K)) \le \dist(b,\,\partial E).
\end{align*}
 Therefore, by applying \eqref{curve 1} with $a=z_1$ and 
 \eqref{curve 2} with $a=b$, the construction of $\gamma_z$ tells 
\begin{align*} 
   \ell(\gamma_z[z,b])=  & \ell(\beta_z[z,\,z_1])+\ell(L_{z_1}[z_1,\,b])\\
    \le &J_0\dist(z_1,\,\partial E)+n\dist(b,\,\partial E)\le (3J_0+n) \dist(b,\,\partial E). 
\end{align*}
This implies \eqref{John consequence} when $b\in L_{z_1}$, and completes the proof. 
\end{proof}

\medskip

 \appendix

 \section{Proof of the Selection Principle.}

\begin{proof}[Proof of Proposition~\ref{selection}]
 Set
 $$A(E):=\min_{x\in\mathbb R^n}|E\Delta( x+K)|$$
    for any measurable set $E$, and observe that this minimum is attained  for some $x\in\mathbb R^n$ whenever $E$ is bounded.
Moreover, for any bounded measurable set $E, F\subset \mathbb R^n$, we have
\begin{equation}\label{triangle index}
    |A(E)-A(F)|\le |E\Delta F| 
\end{equation}
since, by assuming $A(E)\ge A(F)$ without loss of generality and observing that  
$$A(F)=|F\Delta (y+K)| \quad \text{ for some }\ y\in \mathbb R^n,$$
the triangle inequality gives
$$A(E)-A(F)\le |E\Delta (y+K)|-|F\Delta (y+K)|\le |E\Delta F|.$$

\medskip    
\noindent{\bf Step 1:}
    We first claim that, up to translations, $ K$ is the unique minimizer of the problem:
    \begin{equation}\label{isoperi pro 1}
        \min\{P_K(U)+\Lambda\big||U|-|K|\big|:U\subset\mathbb R^n\}\quad \text{ for $\Lambda>n$}.
    \end{equation}
    Indeed, by the Wulff inequality, we may replace $U$ in \eqref{isoperi pro 1} with some $(x+rK)$ satisfying $|(x+rK)|=|U|$. Then \eqref{isoperi pro 1} is equivalent to finding   where the minimum of 
    $$h(r):=nr^{n-1}+\Lambda|r^n-1|$$
    can be reached. Since $h$ has a unique minimum when $r=1$ in case $\Lambda >n$, our claim holds.

\medskip
    \noindent{\bf Step 2: Replace $E_k$ by a $(\Lambda+1, R_0)$-minimizer}. 
    Let
    \begin{equation}\label{epsi r}
        \Lambda=n+1\quad \text{and}\quad R_0=10.
    \end{equation}
    To set up the new set, for every $k\in\mathbb N^+$ we consider a minimizer $F_k'$ of the following problem:
    \begin{equation}\label{minimizer squ}
        \min\Big\{P_K(U)+\big|A(U)-A(E_k)\big|+\Lambda\big||U|-|K|\big|:U\subset  R_0 K\Big\},
    \end{equation}
where $\Lambda>n$ is fixed. Since the variational energy \eqref{minimizer squ} is invariant under translation, we may assume 
\begin{equation}\label{bary}
    \int_{F_k'} x=0
\end{equation}
up to a translation. In addition, up to extracting a subsequence,  the point $y_k\in\mathbb R^n$ satisfying 
\begin{equation}\label{ykfk}
    A(F_k')=|F_k'\Delta (y_k+K)|
\end{equation} 
converges to some point $y_0\in \mathbb R^n$. 
Moreover,
by \cite[Theorem 12.26]{M2012}, up to passing to a subsequence, there is a set $F$ satisfying 
\begin{equation}\label{construct minimizer}
    \int_F x=0,\quad \lim_{k\to +\infty}|F_k'\Delta F|= 0\quad \text{and}\quad P_K(F)\le \liminf_{k\to +\infty}P_K(F_k').
\end{equation}
As a result, recalling that $A(E_k)\to 0$ since  $E_k\to K$ as $k\to +\infty$, we use \eqref{minimizer squ}, \eqref{ykfk} and \eqref{construct minimizer} to obtain the minimality of $F$, that is, for any $U\subset  R_0 K$ with finite perimeter,
\begin{align}\label{minimal of F}
    & P_K(F)+A(F)+\Lambda\big||F|-|K|\big|\le P_K(F)+|F\Delta (y_0+K)|+\Lambda\big||F|-|K|\big|\nonumber\\
    &\le \liminf_{k\to+\infty}P_K(F_k')+|F_k'\Delta (y_k+K)|-A(E_k)+\Lambda\big||F_k'|-|K|\big|\nonumber\\
    &\le \liminf_{k\to+\infty}P_K(F_k')+\big|A(F_k')-A(E_k)\big|+\Lambda\big||F_k'|-|K|\big|\nonumber\\
    &\le \liminf_{k\to+\infty} P_K(U)+\big|A(U)-A(E_k)\big|+\Lambda\big||U|-|K|\big|\nonumber\\
    & = P_K(U)+A(U)+\Lambda\big||U|-|K|\big|.
\end{align}
Further recalling \eqref{isoperi pro 1} and that, up to a translation, $K$ is the unique solution of 
$$\min\{A(U):U\subset  R_0 K\},$$ 
 it implies that,  up to translations, $K$ is the unique solution of 
$$\min\{ P_K(U)+A(U)+\Lambda\big||U|-|K|\big|:U\subset  R_0 K\}.$$
As a result, from \eqref{minimal of F}, \eqref{construct minimizer} and \eqref{ykfk} it follows that  $F=K$ with $\int_{K} x =0$.

    Observe that $\partial F_k'$ converges to $\partial K$ in the Hausdorff metric since \eqref{construct minimizer} and $F_k'\subset R_0 K $. Next, we show that  $F_k'$ are all $(\Lambda+1, R_0)$-minimizers.  

    To this aim, we choose $x\in F_k'$ and $r\in(0, R_0)$  and a set $U$ satisfying
    $F_k'\Delta U\subset\subset (x+rK)$. Then we have two cases:

\medskip
\noindent{\bf Case 1: $U\subset R_0 K$.} 
By using \eqref{minimizer squ} and \eqref{triangle index} in sequence,  we have
\begin{align}\label{ezr minimizer}
    P_K(F_k')& \le P_K(U)+\big|A(U)-A(E_k)\big|-\big|A(F_k')-A(E_k)\big|+\Lambda\big||U|-|K|\big|-\Lambda\big||F_k'|-|K|\big|
    \nonumber\\
    & \le P_K(U)+|U\Delta F_k'|+\Lambda\big||U|-|F_k'|\big|\le P_K(U)+(1+\Lambda)|U\Delta F_k'|.
\end{align}

\medskip
\noindent{\bf Case 2: $|U\setminus  R_0 K|>0$.} 
Let $U':=U\cap R_0 K$. Since $F_k'\cup U'\subset R_0 K$ holds from the definition of $U'$ and \eqref{minimizer squ}, it leads to $|F_k'\setminus U'|=|F_k'\setminus U|$ so that 
$$|U'\Delta F_k'|= |U'\setminus F_k'|+|F_k'\setminus U'|\le |U\setminus F_k'|+|F_k'\setminus U|= |U\Delta F_k'|.$$
Consequently, again using $U'\subset R_0 K$,  we obtain from the consequence of case 1 that  
\begin{equation}\label{return case1}
        P_K(F_k')-P_K(U')\le  (\Lambda+1)|U'\Delta F_k'|\le (\Lambda+1)|U\Delta F_k'|.
\end{equation}
Moreover, the  Wulff inequality applied to $U\cup R_0 K$ tells that  there exists  $R>R_0$ with $|U\cup R_0 K|=|R K|$ so that 
\begin{equation}\label{iso estimate}
    P_K(R_0 K)< P_K(R K)\le P_K(U\cup R_0 K).
\end{equation}
This implies
$$P_K(U')-P_K(U)=P_K(R_0 K)-P_K(U\cup R_0 K)< 0,$$
 which, combined with \eqref{return case1}, yields that 
 \begin{align}\label{ezr minimizer 2}
     P_K(F_k')-P_K(U)& =\big(P_K(F_k')-P_K(U')\big)+\big(P_K(U')-P_K(U)\big)\le (\Lambda+1)|U\Delta F_k'|.
 \end{align}
 Thus, we conclude that $F_k'$ are $(\Lambda+1,R_0)$-minimizers.

\medskip
 \noindent{\bf Step 3:  $F_k'$ are John  whenever $k$ is sufficiently large. } 
  Up to passing to a subsequence, for the same  $\delta$ in Lemma \ref{John minimizer}, whenever $k$ is sufficiently large,
 \begin{equation}\label{boundary near K}
     \partial F_k'\subset (1+\delta) K\setminus (1-\delta)K
 \end{equation}
 holds. Hence, we claim that 
 \begin{equation}\label{nearly K}
     (1-\delta)K\subset F_k'\subset (1+\delta) K.
 \end{equation} 
 
 Indeed, the uniform boundedness of $F_k'$ and \eqref{boundary near K} ensure that $F_k' \subset (1+\delta) K$. We further suppose that there is a point $z\in (1-\delta)K$ with $z\notin F_k'$. Then any curve $\gamma\subset (1-\delta)K$ joining $z$ to $\partial (1-\delta)K$ satisfies $\gamma\cap \overline{F_k'}= \emptyset$ from \eqref{boundary near K} and the connectivity of $\gamma$. As a result,  $ (1-\delta) K\subset \mathbb R^n\setminus F_k'$ so that 
 $$0\gets |F_k'\Delta K|\ge |K\setminus F_k'|\ge |(1-\delta)K|$$
 which yields the contradiction. Hence, $(1-\delta)K\subset F_k'$ and then \eqref{nearly K}  holds. 

 As a result, combining \eqref{nearly K}, \eqref{epsi r} and Lemma \ref{John minimizer},  $F_k'$ are all $J$-John domains with $J=J(n)$.  

\medskip
 \noindent{\bf Step 4:  Properly scale $F_k'$.} 
 Recalling \eqref{assumption}, we have
 $$P_K(E_k)-P_K(K)\le \beta_k A(E_k)^2,$$ 
 where $\beta_k\in(0,\delta_k]$, 
since we have taken subsequence for $E_k$ and $F_k'$ multiple times.
 The minimality \eqref{minimizer squ} of $F_k'$, together with \eqref{assumption}, gives
 \begin{align}\label{measure to K}
     & P_K(F_k')+\Lambda\big||F_k'|-|K|\big|+\big|A(F_k')-A(E_k)\big| \le  P_K(E_k)\le P_K(K)+\beta_kA(E_k)^2 
 \end{align}
 and then the minimality \eqref{isoperi pro 1} of $K$ further provides
 \begin{equation}\label{K mini}
     P_K(K)+\beta_k A(E_k)^2\le P_K(F_k')+\Lambda\big||F_k'|-|K|\big|+\beta_kA(E_k)^2.
 \end{equation}
Hence, by comparing the left hand side of \eqref{measure to K} and the right hand side of \eqref{K mini},  we obtain that 
 $\big|A(F_k')-A(E_k)\big|\le \beta_k A(E_k)^2$. This, by gathering the assumption $E_k\Delta K\to 0$ 
 and $\beta_k\to 0$, leads to
 \begin{equation}\label{difference EF}
     \lim_{k\to+\infty}\frac{A(F_k')}{A(E_k)}=1.
 \end{equation}

To achieve the proof we need to scale $F_k'$ appropriately. Assume that $F_k=\lambda_k F_k'$, where $\lambda_k$ satisfies $|F_k|=|K|$. Observe that $\lambda_k\to 1$ holds by  $F_k'\to K$. Then via $$\lim_{k\to +\infty}P_K(F_k')=P_K(K) $$
given by  \eqref{measure to K}, we have $\lim_{k\to +\infty}P_K(F_k)=P_K(K) $. Hence, $P_K(F_k)/|F_k|<\Lambda$ holds whenever $k$ is sufficiently large, since $\lim_{k\to \infty}P_K(F_k)/|F_k|=P_K(K)/|K|=n<\Lambda$. 
Consequently, for $k$ large enough we have 
\begin{align}\label{estimate 3}
    |P_K(F_k)-P_K(F_k')|& =P_K(F_k)|1-\lambda_k^{1-n}|\le P_K(F_k)|1-\lambda_k^{-n}|\nonumber\\
    & <\Lambda|F_k| |1-\lambda_k^{-n}|= \Lambda \big||F_k|-|F_k'|\big|.
\end{align}
which, together with  \eqref{measure to K}, yields that 
\begin{equation}\label{estimate 4}
    P_K(F_k)\le P_K(F_k')
+\Lambda \big||F_k|-|F_k'|\big|= P_K(F_k')
+\Lambda \big||K|-|F_k'|\big|\le P_K(K)+\beta_k A(E_k)^2. 
\end{equation}

Now  we recall that \eqref{difference EF}
implies that $A(E_k)^2< 2A(F_k)^2 $  whenever $k$ is large enough, from which \eqref{estimate 4} gives \eqref{John selection} by setting $\alpha_k:=2\beta_k$. Further observe that  \eqref{bary} yields \eqref{barycond}. Hence,  we  conclude that $F_k$ are the desired $J$-John domains. 
\end{proof}


\begin{thebibliography}{99}

\bibitem{AFM2013}
E. Acerbi, N. Fusco, M. Morini, \emph{Minimality via second variation for
a nonlocal isoperimetric problem}. Communications in Mathematical Physics, 2013, 322(2): 515-557.

\bibitem{ANP2002}
L. Ambrosio, M. Novaga, E.  Paolini, 
\emph{Some regularity results for minimal crystals}.  A tribute to J. L. Lions. 
ESAIM Control Optim. Calc. Var. \textbf{8} (2002), 69--103.
 
\bibitem{AP1999}
 L. Ambrosio, E. Paolini, \emph{Partial regularity for quasi minimizers of perimeter}. Papers in memory of Ennio De Giorgi. Ricerche Mat. 48 (1999), suppl., 167--186.


\bibitem{B1988}
 B. Bojarski,
\emph{Remarks on Sobolev imbedding inequalities}. Complex analysis, Joensuu 1987, 52--68,
Lecture Notes in Math., 1351, Springer, Berlin, 1988.


 
\bibitem{BDV2015}
L. Brasco, G. De Philippis, B. Velichkov, \emph{Faber-Krahn inequalities in sharp quantitative form}, Duke Math.J.
 164 (2015),1777--1831.

\bibitem{BK1995}
S. Buckley, P.  Koskela, \emph{Sobolev-Poincar\'e implies John}.  
Math. Res. Lett. 2 (1995), no. 5, 577--593.



\bibitem{CL2012}
M. Cicalese, G.P. Leonardi, \emph{A selection principle for the sharp quantitative isoperimetric inequality}. Arch. Ration. Mech. Anal. 206, 617-643 (2012).



\bibitem{DS1998}
G. David, S. Semmes,  \emph{Quasiminimal surfaces of codimension $1$ and John domains}. Pacific Journal of Mathematics, 1998, 183(2): 213--277.

%

\bibitem{FM2011}
A. Figalli, F. Maggi, \emph{On the shape of liquid drops and crystals in the small mass regime}. Arch. Ration. Mech. Anal. 201 (2011), no. 1, 143--207.

\bibitem{FMP2008}
N. Fusco, F. Maggi, A.  Pratelli, 
 \emph{The sharp quantitative isoperimetric inequality.} Ann. Math. 168, 941-980 (2008).

\bibitem{FMP2010}
A. Figalli, F. Maggi, A.  Pratelli, 
 \emph{A mass transportation approach to quantitative isoperimetric inequalities.}
Invent. Math. 182 (2010), no. 1, 167--211.

\bibitem{FZ2022}
A. Figalli, Y. Zhang, \emph{Strong stability for the Wulff inequality with a crystalline norm}, Comm. Pure Appl. Math., 75 (2022), no. 2, 422--446.

 \bibitem{F2015}
N. Fusco, \emph{The quantitative isoperimetric inequality and related topics.}
Bull. Math. Sci. 5 (2015), no. 3, 517--607.

\bibitem{FZ2017}
N. Fusco, Y. Zhang, \emph{A quantitative form of the Faber-Krahn inequality}, Calc. Var. Partial Differential Equations 56 (2017), no. 5, 56:138.
%
\bibitem{GR1983}
V. M. Goldshtein, Yu. G. Reshetnyak, 
\emph{Introduction to the theory of functions with generalized derivatives, and quasiconformal mappings} "Nauka'', Moscow, 1983. 285 pp.
 
\bibitem{HK2000}
P. Hajlasz and P. Koskela. \emph{Sobolev met Poincar\'e}, Mem. Amer. Math. Soc.
145, (2000).

%
\bibitem{J1948}
F. John, \emph{Extremum Problems with Inequalities as Subsidiary Conditions}. Studies and Essays
Presented to R. Courant on his 60th Birthday, January 8, 1948, pp. 187--204. Interscience, New York (1948)

\bibitem{KKLS2013}
J. Kinnunen, R. Korte, A. Lorent, N. Shanmugalingam, \emph{Regularity of sets with quasiminimal boundary surfaces in metric spaces.} 
J. Geom. Anal. 23 (2013), no. 4, 1607--1640.

\bibitem{M2012}
F. Maggi, \emph{Sets of finite perimeter and geometric variational problems}. An introduction to geometric measure theory. Cambridge Studies in Advanced Mathematics, 135. Cambridge University Press, Cambridge, 2012.

  \bibitem{M2011}
 V. Maz'ya, \emph{Sobolev spaces: with Applications to Elliptic Partial Differential Equations.} Second, revised and augmented edition. Grundlehren der mathematischen Wissenschaften, 342. Springer, Heidelberg, 2011.  xxviii+866 pp.



\bibitem{NV1991}
R. N\"akki,  J. V\"ais\"al\"a, 
 \emph{John disks}.
Exposition. Math. 9 (1991), no. 1, 3--43.
30C65

\bibitem{N2016}
R. Neumayer, \emph{A strong form of the quantitative Wulff inequality}.  
SIAM J. Math. Anal. 48 (2016), no. 3, 1727--1772.

\bibitem{S1970} 
E. M. Stein,  \emph{Singular integrals and differentiability
properties of functions}.
Princeton University Press, Princeton, New Jersey, 1970.

%


%
\bibitem{V2000}
 J. V\"ais\"al\"a, \emph{Unions of John domains}. Proc. Amer. Math. Soc. 128 (2000), no. 4, 1135--1140.

\bibitem{Z1989}
 W. P. Ziemer, \emph{Weakly differentiable  functions.  Sobolev    spaces    and   functions   of   bounded variation}.  Graduate Texts  in  Mathematics 120, Springer-Verlag, New York, 1989.

\end{thebibliography}
\end{document}